\documentclass[a4paper,twoside]{amsart}
\usepackage{amsmath,amssymb,amsthm,dsfont,graphicx}
\usepackage[colorlinks=true, linkcolor=black, citecolor=blue, urlcolor=red]{hyperref}
\usepackage[english]{babel}
\usepackage[lite]{amsrefs}

\newcommand   {\greeka}      {\alpha                                           }

\newcommand   {\greekb}      {\beta                                            }

\newcommand   {\greekl}      {\lambda                                          }

\newcommand   {\greekp}      {\pi                                              }

\newcommand   {\greekph}     {\varphi                                          }

\newcommand   {\greekO}      {\Omega                                           }
\newcommand   {\scripta}     {{\scalebox{1.10}{$\mathfrak{a}$}}                }

\newcommand   {\scriptf}     {{\scalebox{1.10}{$\mathfrak{f}$}}                }

\newcommand   {\scripth}     {{\scalebox{1.10}{$\mathfrak{h}$}}                }

\newcommand   {\scriptm}     {{\scalebox{1.10}{$\mathfrak{m}$}}                }

\newcommand   {\scriptu}     {{\scalebox{1.10}{$\mathfrak{u}$}}                }

\newcommand   {\setC}        {\mathds{C}                                       }

\newcommand   {\setI}        {\mathds{I}                                       }

\newcommand   {\setN}        {\mathds{N}                                       }

\newcommand   {\setV}        {\mathds{V}                                       }

\newcommand   {\setZ}        {\mathds{Z}                                       }

\newcommand   {\Satisfies}   {\mbox{ satisfies }                               }

\newcommand   {\A}           {\forall \,                                       }
\newcommand   {\E}           {\exists \,                                       }

\newcommand   {\noneq}       {\not =                                           }
\newcommand   {\nonin}       {\not \in                                         }

\newcommand   {\defiff}      {\, :\Longleftrightarrow \,                       }
\newcommand   {\timplies}    {\Rightarrow                                      }

\newcommand   {\tiff}        {\Leftrightarrow                                  }

\newcommand   {\set}[1]      {\left\{\, #1 \,\right\}                          }

\newcommand   {\setin}       {\subseteq \,                                     }

\newcommand   {\setinn}      {\subset                                          }

\newcommand   {\into}        {\, \hookrightarrow \,                            }

\newcommand   {\onto}        {\, \twoheadrightarrow \,                         }

\newcommand   {\onetoone}    {\, \longleftrightarrow \,                        }

\newcommand   {\isoto}       {\, \stackrel{\raisebox{-0.5ex
                                 }[0pt]{$\sim$}}{\longrightarrow} \,           }

\newcommand   {\abs}[1]      {| #1 |                                           }
\newcommand   {\gen}[1]      {\langle #1 \rangle                               }

\newcommand   {\Max}[1]      {\mbox{\rm{max}} \hspace{-0.06cm}
                              \left\{ \, #1 \, \right\}                        }

\newcommand   {\kn}          {\mbox{\rm{ker}}                                  }

\newcommand   {\ideal}       {\unlhd                                           }
\newcommand   {\Modulo}[2]   {\raisebox{3pt}{$#1$} \hspace{-1pt}
                              \raisebox{-1pt}{\scalebox{1.5}{/}}\hspace{1dd}
															\hspace{-2pt} \raisebox{-3pt}{$#2$}              }

\newcommand   {\smax}        {\mbox{\rm{Smax}}                                 }

\newcommand   {\trdeg}       {\mbox{\rm{trdeg}}                                }

\newtheorem{theorem}{Theorem}[section]
\newtheorem{proposition}[theorem]{Proposition}
\newtheorem{corollary}[theorem]{Corollary}

\theoremstyle{definition}
\newtheorem{definition}[theorem]{Definition}

\newtheorem{remark}[theorem]{Remark}
\allowdisplaybreaks

\title[Infinite Versions of Hilbert's Nullstellensatz]{Infinite Versions of Hilbert's Nullstellensatz}

\author[A. Bernhard Zeidler]{A.~Bernhard Zeidler}

\address{Mathematisches Institut, Universit\"at T\"ubingen,
Auf der Morgenstelle 10, 72076 T\"ubingen, Germany}
\email{zeidler@math.uni-tuebingen.de}


\begin{document}

\begin{abstract}
 We compile a long list of equivalent formulations of Hilbert's Nullstellensatz in
 infinite dimensions, and prove a persistence result for the strong Nullstellensatz
 in large polynomial rings.
\end{abstract}

\maketitle

\section{Presentation of the Results}
\label{sec:maintheorem}

 This note is devoted to Hilbert's Nullstellensatz for polynomial rings of infinite
 Krull dimension. We collect classical and recent equivalences (see Theorem~\ref{thm:nullstellensatz})
 and provide complete, elementary proofs (see Section~\ref{sec:alltheproofs}) of which
 some are new or generalize existing ones. Moreover, given an ideal in a polynomial
 ring of dimension lower than the cardinality of the base field, we show that its
 extensions to larger polynomial rings satisfy the strong Nullstellensatz, even if
 their cardinality exceeds that of the base field (see Corollary~\ref{cor:idealfrombelow}).

 Let us fix the setting and notation: All rings and algebras are assumed to be commutative,
 non-zero, with unit and their homomorphisms are assumed to preserve the unit element.
 Given a ring $F$ and a (not necessarily finite) set $I$, let $A := F[t_i \mid i \in I]$
 be the polynomial algebra over $F$. The set of maximal ideals of~$A$ is denoted by
 $\smax(A)$. Any ideal $\scripta \ideal A$ defines the zero set $\setV(\scripta) \setin
 F^I$ and any $X \setin F^I$ defines the vanishing ideal $\setI(X) \ideal A$. The
 vanishing ideal of a single element $x \in F^I$ is denoted by $\scriptm_x$. 

 \begin{theorem} \label{thm:nullstellensatz}
  Let $F$ be a non-zero ring, $I$ be a non-empty index set and the polynomial algebra
	$A = F[t_i \mid i \in I]$. Then the following statements are equivalent:
	\begin{enumerate}
	 \item[(a)]
	  $F$ is an algebraically closed field and $\abs{I} < \abs{F}$ holds for the
		cardinalities.
	 \item[(b)]
	  Any proper ideal $\scripta \ideal A$ admits an $F$-algebra homomorphism $A/\scripta \to F$.
	 \item[(c)]
    Any maximal ideal $\scriptm \ideal A$ admits an $F$-algebra homomorphism $A/\scriptm \to F$.
	 \item[(d)]
	  Any maximal ideal $\scriptm \ideal A$ is of the form $\scriptm = \scriptm_x$ for
		some $x \in F^I$.
	 \item[(e)]
	  The assignment $x \mapsto \scriptm_x$ defines a bijection $F^I \to \smax(A)$.
	 \item[(f)]
    $F$ is a field and any $\scriptm \in \smax(A)$ and $i_1,\dots,i_n \in I$ admit
		$x_{i_k} \in F$ with
		\[ \scriptm \cap F[t_{i_1}, \ldots,  t_{i_n}] \ = \
		   \gen{t_{i_1} - x_{i_1}, \, \dots, \, t_{i_n} - x_{i_n}}. \]
	 \item[(g)]
	  $F$ is a field and for any ideal $\scripta \ideal A$ we get $\setV(\scripta) = \emptyset
		\tiff \scripta = A$.
	 \item[(h)]
	  $F$ is a field and any field extension $F \setin K = F[u_i \mid i \in I]$ is trivial.
	 \item[(i)]
	  For any ideal $\scripta \ideal A$ we have $\sqrt{\scripta} = \setI \setV(\scripta)$.
	\end{enumerate}
 \end{theorem}

 The presumably most popular equivalence, (a) $\tiff$ (d), is attributed to Kaplansky,
 apparently without concrete reference (compare~\cite[p.~407, footnote 1]{Lang}). We
 provide a direct proof here. Statement (b) is the chromatic version \cite[Thm.~6.1]{BSY}
 by Burklund, Schlank and Yuan. We present a slight modification of their proof, to match
 our formulation. Statement (h) is taken from \cite{Lang} and (i) is gained from (d)
 classically via \emph{Rabinowitsch's trick}. We neither saw (c) nor (f) before, but
 they're elementary. The rest is common folklore.
 
 \begin{remark}
  It is clear, that the strong Nullstellensatz \ref{thm:nullstellensatz}.(i) gives
	rise to a one-to-one correspondence between radical ideals of $A = F[t_i \mid i
	\in I]$	and algebraic subsets of $F^I$. All we need is $F$ being an algebraically
	closed field and $\abs{I} < \abs{F}$.
 \end{remark}

 \begin{remark} \label{rem:Oempty}
  Note that we \emph{assumed} $F$ to be a field in (a), (f), (g) and (h) only. In
	the other statements, this is implied. In fact it is possible to drop this assumption
	in (f) as well, by allowing $n = 0$. In this case we have $\scriptm \cap F = 0$ for
	any $\scriptm \in \smax(A)$ and this is equivalent to $F$ being a field.
 \end{remark}

 \begin{remark} \label{rem:lang}
  In 1952 S.~Lang proved a similar equivalence: If $F$ is an algebraically closed
	field and $P \setin F$ is the prime field of $F$, then property (d) holds true
	if and only if $I$ is finite or $\abs{I} < \trdeg(F:P)$. The latter is much the
	same as (a), since the transcendence degree satisfies
	\[ \abs{F} \ = \ \Max{\abs{\setN}, \, \trdeg(F:P)} \]
 \end{remark}

 \begin{corollary} \label{cor:idealfrombelow}
  Let $F$ be an algebraically closed field and $I \setin J$ be index sets, such that
	$I$ is non-empty and $\abs{I} < \abs{F}$. If now $\scriptu \ideal F[t_j \mid j \in J]$
	is an ideal of the form $\scriptu = \scripta \, F[t_j \mid j \in J]$ for some ideal
	$\scripta \ideal F[t_i \mid i \in I]$, then we still get the identity
  \[ \setI \, \setV (\scriptu) \ = \ \sqrt{\scriptu} \]
 \end{corollary}

\section{Proof of Theorem~\ref{thm:nullstellensatz} and Corollary~\ref{cor:idealfrombelow}}
\label{sec:alltheproofs}

 It is well-known, that: If $F$ is a field and $I$ is a non-empty index set, then
 the ideal $\scriptm_x	= \setV(x) = \gen{t_i - x_i \mid i \in I}$ is a maximal ideal
 of $A = F[t_i \mid i \in I]$ and the following map is well-defined and injective
 \[ F^I \, \into \, \smax(A) \; : \; x \, \mapsto \, \scriptm_x \]
 \begin{definition}
	We say that $A$ is {\bf geometric}, iff this map is surjective. That is iff every
	maximal ideal $\scriptm \in \smax(A)$ of $A$ is of the form $\scriptm = \scriptm_x$
	for some $x \in F^I$.
 \end{definition}

 \begin{proposition} \label{pro:geometric}
  Let $F$ be a field, $I$ a non-empty index set and $A = F[t_i \mid	i \in I]$ the
	$F$-algebra of polynomials. Then we get the following statements
  \begin{enumerate}
	 \item[(i)]
	  If $I$ is finite and $F$ is an algebraically closed field, then $A$ is geometric.
	 \item[(ii)]
	  If $I$ is infinite, then $A$ is geometric if and only if $A[s]$ is geometric.
  \end{enumerate}
 \end{proposition}

 In this proposition (i) is the classic (weak) Nullstellensatz, of which numerous
 proofs exist: A short, elegant proof has been given by Zariski in \cite{Zari} using
 the Noether Normalisation theorem. Krull and Goldman introduced the theory of Jacobson
 resp.~Hilbert rings, which can be seen as a generalization of this, e.g.~refer to
 \cite{Clark}. In 1999 Munshi reduced the techniques of Jacobson rings to a minimum
 - refer to \cite{May}, in 2006 Arrondo gave an astoundingly simple proof in \cite{Arro}.
 And in 2018 a proof using Gröbner bases has been provided by Glebsky and Rubio-Barrios 
 \cite{Gleb}. Statement (ii) is quite obvious, we took it from \cite{Lang}. 

 Let us split the proof of Theorem~\ref{thm:nullstellensatz} into parts. The first
 insight is due to \cite{BSY}: {\bf Claim 1:} Under assumption (b), $F$ is an algebraically
 closed field. \emph{Proof of claim 1:} Let us first prove, that $F$ is a field:
 Consider any $a \in F$ with $a \nonin F^{\ast}$ and define the ideal $\scripta :=
 \gen{t_i \mid i \in I} + a A$. Then the homomorphism
 \[ A \, \onto \, \Modulo{F}{a F} \; : \; f \, \mapsto \, f(0) + a F \]
 is surjective and has kernel $\scripta$, which is a proper ideal, since $a \nonin
 F^{\ast}$. Therefore $A/\scripta$ and $F/a F$ are isomorphic as $F$-algebras and
 by (b) there is a homomorphism $\greekph$ of the form $\greekph : A/\scripta \to
 F$. This gives rise to a homomorphism of $F$-algebras
 \[ \Modulo{F}{a F} \, \isoto \, \Modulo{A}{\scripta} \, \to \, F \]
 Explicitly this is $F/a F \to F : b + a F \mapsto b 1_A + \scripta \mapsto \greekph(
 b 1_A + \scripta) = b$. But from $a + a R \mapsto a$ and $a + a R \mapsto 0$ we find
 $a = 0$ and therefore $F$ is a field.

 Next we have to show, that $F$ even is algebraically closed: Fix any $j \in I$, let
 $H := I \setminus \set{j}$ and consider an arbitrary, non-constant polynomial $f \in
 F[t_j] \setin A$. Now define the ideals $\scripth := \gen{t_i \mid i \in H}$ and
 $\scripta := f A + \scripth$. Then it is clear, that $F[t_j] \isoto A/\scripth : g
 \mapsto g + \scripth$ is an isomorphism of $R$-algebras, that maps $f F[t_j]$ to
 $\scripta/\scripth$. And therefore we find the isomorphism
 \[ \Modulo{F[t_j]}{f F[t_j]} \; \isoto \; \Modulo{A/\scripth}{\scripta/\scripth}
	  \; \isoto \; \Modulo{A}{\scripta} \; : \; g + f F[t_j] \, \mapsto \, g + \scripta \]
 Now $F[t_j]/f F[t_j]$ is a $\deg(f)$-dimensional $F$-vector space and as $f$ is
 non-constant, we have $\deg(f) \geq 1$. In particular $F[t_j] / f F[t_j] \noneq
 0$ and therefore $A / \scripta \noneq 0$. That is $\scripta$ is a proper ideal
 and by (b) there is a homomorphism $\greekph$ of $F$-algebras
 \[ \greekph \; : \;  \Modulo{F[t_j]}{f F[t_j]} \; \isoto \;
	  \Modulo{A}{\scripta} \; \onto \; F \]
 Let $\scriptf := f F[t_j]$ and $x := \greekph \big( t_j + \scriptf \big) \in F$.
 As $\greekph$ is a homomorphism of $R$-algebras, it commutes with polynomials and
 hence $f(x) = \greekph \big( f(t_j + \scriptf) \big) = \greekph \big(f + \scriptf
 \big) = 0_F$. Thus any non-constant polynomial over $F$ has a root.
 \qed \\

 {\bf Claim 2:} Statements (b) to (e) are equivalent. \emph{Proof of claim 2.} (b)
 $\timplies$ (c) is trivial, for the converse implication choose a maximal ideal
 $\scriptm$ containing $\scripta$. Then we can compose $A/\scripta \onto	A/\scriptm
 \to F$. In (d) $\timplies$ (c) we are given $\scriptm = \scriptm_x$ for some $x \in
 F^I$. Then $\greekph : f \mapsto f(x)$ is the homomorphism we seek. Conversely (c)
 implies, that $F$ is a field, via (b) and claim 1. Given a maximal ideal $\scriptm$
 fix $\greekph : A/\scriptm \to F$ and $x := (x_i)$ where $x_i := \greekph(t_i + \scriptm)
 \in F$. Then $\scriptm_x \setin \kn(\greekph) = \scriptm$. Over a field $\scriptm_x$
 is maximal, thus $\scriptm_x = \scriptm$, which is (d). Next (e) $\timplies$ (d)
 is trivial and conversely we already know (d) implies (b) such that $F$ is a field
 again. Hence $x \mapsto \scriptm_x$ is injective, now (d) is the surjectivity missing
 for (e). 
 \qed \\

 {\bf Claim 3:} (d) $\timplies$ (a) \emph{ Proof of claim 3.} By claim 1 it only remains
 to prove $\abs{I} < \abs{F}$. If $I$ is finite, then $\abs{I} < \abs{F}$ is clear, as
 an algebraically closed field always is infinite. So from now on we assume, that $I$
 is infinite and elaborate on an argument given in \cite{Lang}: Fix any $j \in I$,
 as $I$ is infinite we have $\abs{I \setminus \{j\}} = \abs{I}$ and thereby there
 is a subset $H \setin I \setminus \{j\}$ such that $H \onetoone F : i \mapsto a_i$.
 We now define a homomorphism of $F$-algebras $\greekph : A \to F(s)$, by expansion of
 \begin{eqnarray*}
	\greekph(t_i) \, := \, s                 & \mbox{for} &	i \, =   \, j                \\
	\greekph(t_i) \, := \, \frac{1}{s - a_i} & \mbox{for} &	i \, \in \, H                \\
	\greekph(t_i) \, := \, 0                 & \mbox{for} &
	i \, \in \, I \setminus \Big( H \cup \{j\} \Big)
 \end{eqnarray*}
 We first have to check, that $\greekph$ is surjective: As $F$ is algebraically
 closed any polynomial $p(s) \in F[s]$ splits into linear factors, but as we have
 enumerated $F$ using the index set $H$, we may write it in the form
 \[ p(s) \ = \ \greekl \, (s - a)^{\greeka} \ := \
    \greekl \, \prod_{i \in H} (s - a_i)^{\greeka_i} \]
 with only finitely many $\greeka_i \in \setN$ being non-zero. Therefore any rational
 function $p/q$, where $p$, $q \in F[s]$ and $q \noneq 0$, can be written in the
 following form, where $k_i \in \setZ$ and only finitely many $k_i$ are non-zero
 \[ \frac{p(s)}{q(s)} \ = \ \greekl \, \prod_{i \in H} (s - a_i)^{k_i} \]
 By construction of $\greekph$ it is clear, that any polynomial $f(t_j) \in F[t_j]
 \setin A$ is simply mapped to $f(s) \in F(s)$. Likewise any monomial $t^{\greekb}
 \in F[t_i \mid i \in H] \setin A$ where $\greekb \in \setN^{\oplus H}$ is mapped
 to $1/(s-a)^{\greekb}$. Given $p/q \in F[s]$, write $p/q$ in the form above and
 pick up $\greeka$ and $\greekb \in \setN^{\oplus H}$ such that $k_i = \greeka_i
 - \greekb_i$ for any $i \in H$ (and $\greeka_i = 0 = \greekb_i$ for $k_i = 0$).
 Then we find polynomials $f := \greekl \prod_i (t_j - a_i)^{\greeka_i} \in F[t_j]
 \setin A$ and $g := t^{\greekb} \in F[t_i \mid i \in H] \setin A$ such that $fg
 \in A$ is mapped to $p/q \in F(s)$ under $\greekph$
 \[ \greekph \left( \greekl \prod_{i \in H} (t_j - a_i)^{\greeka_i} \cdot
    t^{\greekb} \right) \; = \; \greekl \prod_{i \in H} (s - a_i)^{\greeka_i}
 	 \cdot \frac{1}{(s-a)^{\greekb}} \; = \; \greekl \prod_{i \in H} \frac{(s
	 - a_i)^{\greeka_i}}{(s - a_i)^{\greekb_i}} \; = \; \frac{p(s)}{q(s)} \]
 As $p/q$ have been arbitrary $\greekph$ is surjective and if we let $\scriptm :=
 \kn(\greekph)$ we therefore find an isomorphism $A/\scriptm \isoto F(s)$. As $F(s)$
 is a field, $\scriptm$ is a maximal ideal. Now by (d) we find $\scriptm = \scriptm_x$
 for some $x \in F^I$, in particular $t_i - x_i \in \scriptm$ for some $i \in H$.
 By construction this means
 \[ 0 \; = \; \greekph(t_i - x_i) \; = \; \frac{1}{s - a_i} - x_i
      \; = \; \frac{1 - x_i s + x_i a_i}{s - a_i} \]
 This yields $1 - x_i s + x_i a_i = 0$ and by comparing the coefficients of this
 polynomial in $s$, we get $x_i = 0$ and $0 = 1 + x_i a_i = 1$. But $0 = 1$ is an
 obvious contradiction to the assumption $\abs{F} \leq \abs{I}$.
 \qed \\

 {\bf Claim 4:} (d) and (f) are equivalent. \emph{Proof of claim 4.} Starting with
 (d) we also have (b) and hence $F$ is a field. Now, if $\{i_1,\dots,i_n\} \setin I$
 and $\scriptm \ideal A$ is a maximal ideal, then $\scriptm = \scriptm_x$ for some
 $x \in F^I$, by (d). In particular we have $t_{i_k} - x_{i_k} \in \scriptm$ and
 $t_{i_k} - x_{i_k} \in A_0 := F[t_{i_1},\dots,t_{i_n}]$ is clear, for any $1 \leq
 k \leq n$. Thereby
 \[ \gen{t_{i_1} - x_{i_1},\dots,t_{i_n} - x_{i_n}} \ \setin \ \scriptm \cap A_0 \]
 As $F$ is a field $\gen{t_{i_1} - x_{i_1},\dots,t_{i_n} - x_{i_n}}$ is a maximal ideal
 of $A_0$. Yet $\scriptm \cap A_0$ is a proper (even prime) ideal and therefore this
 inclusion turns into an identity.
		
 Conversely for (f) $\timplies$ (d) we start with a maximal ideal $\scriptm$ and have
 to find some $x \in F^I$, such that $\scriptm = \scriptm_x$. Fix an arbitrary $i \in
 I$, by (f) we have $\scriptm \cap F[t_i] = \gen{t_i - x_i}$ for some $x_i \in F$. As
 we can do this with any $i \in I$, we find a point $x = (x_i) \in F^I$ with $t_i -
 x_i \in \scriptm$ for any $i \in I$. That is $\scriptm_x \setin \scriptm$. As $F$ is
 a field $\scriptm_x$ is a maximal ideal such that $\scriptm_x \setin \scriptm$ yields
 $\scriptm_x = \scriptm$. 

 To prove remark \ref{rem:Oempty} Let us also consider $n = 0$: In this case $F[\emptyset
 ] = F$ and $\gen{\emptyset} = 0$. Thus (f) reads as $\scriptm \cap F = 0$ for any
 $\scriptm \in \smax(A)$. If $F$ is a field, then $\scriptm \cap F = 0$ is clear, as
 $0$ is the only proper ideal of $F$. Conversely assume $\scriptm \cap F = 0$ for any
 $\scriptm \in \smax(A)$, we need to show that $F$ is a field: Consider any $a \in F
 \setminus F^{\ast}$. Then $a F \ideal F$ is a proper ideal and hence there is a maximal
 ideal $\scriptm_0 \in \smax(F)$ containing $a F \setin \scriptm_0$. Then $\scriptm :=
 \set{f \in A \mid f(0) \in \scriptm_0}$ is a maximal ideal of $A$. Yet $a \in \scriptm_0
 = \scriptm \cap F = 0$ implies $a = 0$ and as $a$ has been arbitrary we conclude $F
 \setminus \{0\} = F^{\ast}$, which means $F$ is a field.
 \qed \\

 Let us skip some: The equivalence of (d) and (g) is a classic, the proof is plain
 to see. The equivalence of (g) and (h) is taken from \cite{Lang}, we do not relay
 it here, except (h) $\timplies$ (b) is immediate: Let $K := A/\scriptm$ and $u_i
 := t_i + \scriptm$, then $F[U] = A/\scriptm$ and in particular $A/\scriptm \to F$.
 The direction (i) $\timplies$ (g) is immediate and the converse can now be done by
 \emph{Rabinowitsch's trick}, which is based on Proposition \ref{pro:geometric}.(ii).
 This brings (g) to (i) into the list of equivalences. \\

 \goodbreak

 It only remains to tackle {\bf Claim 5:} (a) $\timplies$ (d). \emph{Proof of Claim
 5.} If $I$ is finite, then $A = F[t_i \mid i \in I]$ is geometric, by \ref{pro:geometric}.(i)
 and this is (d). So we only have to consider infinite $I$, for which we generalize
 a proof found in \cite{Brod} for $F = \setC$: By definition any polynomial $f \in
 A = F[t_i \mid i \in I]$ is an $F$-linear combination over the monomials $t^{\greeka}$
 where $\greeka \in \setN^{\oplus I}$. However, as $I$ is infinite
 \[ \abs{\setN^{\oplus I}} \ = \ \Max{\abs{\setN}, \, \abs{I}} \ = \ \abs{I} \]
 We can choose any well-ordering on $I$ and use the bijection $I \onetoone \setN^{
 \oplus I} : i \mapsto \greeka_i$ to install a total order $\greeka_i \leq \greeka_j
 \defiff i \leq j$. We do not require this to be a monomial order, nevertheless we
 may define the degree of a polynomial $f \in A$, as usual
 \[ \deg(f) \ := \ \Max{\greeka \in \setN^{\oplus I} \mid
	  \mbox{$\greeka$-coefficient of $f$} \noneq 0} \]
 If now $\greekb \in \setN^{\oplus I}$ is any multi-index then $\set{t^{\greeka}
 \mid \greeka \leq \greekb}$ corresponds to a subset of $I$. Hence the subspace
 of $A$ generated by this set has dimension $\abs{I}$, at most:
 \[ \greekO_{\greekb} \ := \ \set{f \in A \mid \deg(f) \leq \greekb} 
    \ \leq \ A \]
 Now consider some maximal ideal $\scriptm \ideal A$. We need to show that $\scriptm$
 belongs to some point $x = (x_i) \in F^I$. By maximality of $\scriptm_x$, it suffices
 to show $\scriptm_x \setin \scriptm$, which is
 \[ \A i \in I \ \E x_i \in F \ : \ t_i - x_i \, \in \, \scriptm \]
 Assume this was false, i.e.~there was some $j \in I$ such that for any $z \in F$ we
 would get $t_j - z \nonin \scriptm$. Formally the assumption for this $j$ reads as
 \[ \E j \in I \ \A z \in F \ : \ t_j - z \, \nonin \, \scriptm \]
 As $\scriptm$ is maximal this implies $\scriptm + (t_j - z) A = A$ and hence (for
 any $z \in F$) we may find polynomials $m^z \in \scriptm$ and $g^z \in A$ such that
 \[ \A z \in F \ : \ m^z + (t_j - z) g^z \, = \, 1 \]
 Every polynomial $g^z$ has a degree and hence $F$ may be decomposed into a disjoint
 union of the following subsets
 \[ F \ = \ \bigsqcup_{i \in I} \set{z \in F \mid \deg(g^z) = \greeka_i} \]
 If all these sets $F_i := \set{z \in F \mid \deg(g^z) = \greeka_i}$ would satisfy
 $\abs{F_i} \leq \abs{I}$ then there would be surjective maps $s_i : I \onto F_i$
 and thereby a surjective map $I \times I \onto F$ by $(h,i) \mapsto s_i(h)$. Yet as
 $I$ is infinite we have $\abs{I \times I} = \abs{I} < \abs{F}$ by assumption and
 hence there cannot be a surjective map from $I \times I$ onto $F$. We conclude
 that there must be some $i \in I$ such that $\abs{F_i} > \abs{I}$. Let $\greekb
 := \greeka_i$ and $Z := F_i$ for this $i \in I$
 \[ Z \ = \ \set{z \in F \mid \deg(g^z) = \greekb} \ \Satisfies \ \abs{Z} \, > \, \abs{I} \]
 As we have seen $\greekO_{\greekb}$ is an $F$-vectorspace of dimension at most
 $\abs{I}$. If some of the $g^z$ where $z \in Z$ were identical, then we could
 skip the next step and continue with the next linear relation, right away. If
 all the $g^z$ where $z \in Z$ are distinct, then the set $\set{g^z \mid z \in Z}$
 has cardinality $\abs{Z} > \abs{I} \geq \dim \big( \greekO_{\greekb} \big)$.
 Thus the $g^z$ have to satisfy a linear relation, that is there are pairwise
 distinct elements $z_1,\dots, z_r \in Z$ and coefficients $\greekl_1,\dots,
 \greekl_r \in F^{\ast}$ such that
 \[ \sum_{p = 1}^r \greekl_p \, g^{z_p} \ = \ 0 \]
 We now define a specific polynomial $g \in F[t_j] \setin A$ in the one variable
 $t_j$ over $F$
 \[ g \ := \ \sum_{p = 1}^r \greekl_p \prod_{q \noneq p} (t_j - z_q)
	    \ \in \ F[t_j] \ \setin \ A \]
 By construction we have $g(z_1) = \greekl_1 (z_1 - z_2) \dots (z_1 - z_r) \noneq 0$
 and hence $g$ is non-zero itself. We now claim that this polynomial is nevertheless
 contained in the ideal $\scriptm$
 \begin{eqnarray*}
  g & = & \sum_{p = 1}^r \greekl_p \prod_{q \noneq p} (t_j - z_q)                      \\
    & = & \sum_{p = 1}^r \greekl_p [m^{z_p} + (t_j - z_p) g^{z_p}]
          \prod_{q \noneq p} (t_j - z_q)                                               \\
    & = & \sum_{p = 1}^r \left[ \greekl_p \prod_{q \noneq p} (t_j - z_q) \right]
	        m^{z_p} + \left[ \sum_{p = 1}^r \greekl_p \, g^{z_p} \right]
					\prod_{q = 1}^r (t_j - z_q)                                                  \\
    & = & \sum_{p = 1}^r \left[ \greekl_p \prod_{q \noneq p} (t_j - z_q) \right]
		      m^{z_p} + 0 \ \in \ \scriptm
 \end{eqnarray*}
 As $g \noneq 0$ and $g \in F^{\ast}$ is ruled out by $g \in \scriptm$, we see $g
 \nonin F$. But $\scriptm$ is maximal, in particular prime, and hence some prime
 factor of $g \in F[t_j]$ is contained in $\scriptm$. But as $F$ is algebraically
 closed, the prime factors of $g$ are of the form $t_j - z$ for some $z \in F$, in
 contradiction to the assumption that $t_j - z \nonin \scriptm$ for any $z \in F$.
 \qed \\

 \begin{proof}[Proof of Corollary~\ref{cor:idealfrombelow}]
 Let us abbreviate $F[I] := F[t_i \mid i \in I]$, respectively $F[J] := F[t_j \mid
 j \in J]$. Also denote the canonical projection $\greekp_I^J : F^J \onto F^I : x
 \mapsto x_I$ where for $x = (x_j)_{j \in J}$ we let $x_I := (x_i)_{i \in I}$. Then
 it is straightforward to see, that
 \[ \setV \Big( \scripta \, F[J] \Big) \ = \ 
	  \Big( \greekp_I^J \Big)^{-1} \Big( \setV \big( \scripta \big) \Big) \]
 We need to prepare another identity: {\bf Claim 6:} For any subset $X \setin F^I$
 we have
 \[ \setI \left( \Big( \greekp_I^J \Big)^{-1}(X) \right) \ = \ \setI(X) \, F[J] \]
 The cases $I = J$ and $X = \emptyset$ are trivial, so consider $I \setinn J$ and $X
 \noneq \emptyset$. In this case let $H := J \setminus I$ and split $x = (x_j) \in
 F^J$ into $x = (x_I,x_H) \in F^I \times F^H$ and likewise $(t_j)$ into $(t_I,t_H)$.
 That is we write $f(t) \in F[J]$ in the form $f(t_I)(t_H) \in F[I][H]$. Any $f \in
 \setI(X) \, F[J]$ is generated by some $a \in \setI(X)$ and as $a(x) = a(x_I) = 0$
 for any $x \in \widehat{X} := \big( \greekp_I^J \big)^{-1}(X) = X \times F^H$ this
 implies $f(x) = 0$, as well. If we conversely start with $f \in \setI(\widehat{X})$,
 then $0 = f(x) = f(x_I)(x_J)$ for any $x \in X \times F^H$. That is for any $x_I
 \in X$ we get $f(x_I)(t_H) \in \setI(F^H)$. And as $F$ is an infinite integral domain
 and $H \noneq \emptyset$ this means $f(x_I)(t_H) = 0 \in F[H]$. As $x_I \in X$ was
 arbitrary $f$ has to be an $F[H]$-linear combination of polynomials $a \in \setI(X)$
 and this means $f \in \setI(X) \, F[J]$. 
		
 Now that we have claim 6, we continue: The inclusion $\sqrt{\scriptu} \setin \setI \,
 \setV (\scriptu)$ is generally true, for any reduced ring. For the converse inclusion
 note that by the above we have $\setV(\scriptu) = \big( \greekp_I^J \big)^{-1}(\setV(
 \scripta))$ and thereby
 \begin{eqnarray*}
  \setI \, \setV (\scriptu)
  &  =  & \setI \left( \Big( \greekp_I^J \Big)^{-1} \left( \setV(\scripta) \right) \right)
  \; = \; \setI \, \setV(\scripta) \, F[J]                                             \\
  &  =  & \sqrt{\scripta} \, F[J]
	\; \setin \; \sqrt{\scripta \, F[J]}
	\; =  \; \sqrt{\scriptu}
 \end{eqnarray*}
 \end{proof}

\goodbreak

\end{document}